\documentclass[11pt]{amsart}
\usepackage{comment}
\usepackage{amsmath,amsfonts,amssymb,amstext,amsthm,graphicx,fancyhdr,mathtools,epstopdf}
\usepackage{color}
\usepackage[top=1in, bottom=1in, left=1in, right=1in]{geometry}
\usepackage[titletoc,title]{appendix}
\usepackage{bbm}
\usepackage{bm}
\usepackage{dsfont} %
\usepackage[english]{babel}
\usepackage[T1]{fontenc}
\usepackage{latexsym}
\usepackage{mathrsfs}
\usepackage{graphics}

\usepackage{enumerate, paralist} %
\usepackage{color}
\usepackage{fullpage}

\usepackage[all,cmtip]{xy}

\DeclareFontFamily{OT1}{pzc}{}
\DeclareFontShape{OT1}{pzc}{m}{it}{<-> s * [1.10] pzcmi7t}{}
\DeclareMathAlphabet{\mathpzc}{OT1}{pzc}{m}{it}

\numberwithin{equation}{section}

\theoremstyle{plain}
  \newtheorem{theorem}{Theorem}[section]
  
  \newtheorem{lemma}[theorem]{Lemma}

\theoremstyle{definition}

\newtheorem{example}[theorem]{Example}

\theoremstyle{remark}
\newtheorem{remark}[theorem]{Remark}

{}%
{}%
{\bfseries}%
{}%
\usepackage{color}
\usepackage{amsmath,amsthm}
\usepackage{graphicx}

\numberwithin{equation}{section}

\usepackage{amsmath}
\usepackage{amssymb}
\usepackage{amsthm}
\usepackage{url}
\usepackage{color}
\usepackage{dsfont}
\usepackage{epsfig}
\usepackage[hidelinks]{hyperref}
\usepackage{setspace}
\usepackage{comment}
\usepackage{booktabs}
\usepackage{float}
\usepackage{tikz}
\numberwithin{equation}{section}
\usepackage{enumitem}

\usepackage{geometry}

\renewcommand{\1}{\mathds{1}}
\newcommand{\mfo}{\mathfrak{o}}

\renewcommand{\l}{\mathfrak{l}}
\renewcommand{\r}{\mathfrak{h}}

\newcommand{\Hqy}{\mathbf{H}^{[y}}
\newcommand{\Hqb}{\mathbf{H}^{[b}}

\newcommand{\con}{\mathfrak{C}}
\newcommand{\Hnormy}{\mathscr{H}^{[y}}
\newcommand{\Hnormb}{\mathscr{H}^{[b}}
\newcommand{\Hnormz}{\mathscr{H}^{[z}}
\newcommand{\scale}{W}

\title{Fluctuation theory of continuous-time skip-free downward Markov chains with applications to branching processes with immigration}

\author{ R. Loeffen}
\address{University of Manchester, School of Mathematics,  Manchester, M13 9PL, UK.}
\email{ronnie.loeffen@manchester.ac.uk}

\author{P. Patie}
\address{School of Operations Research and Information Engineering, Cornell University, Ithaca, NY 14853.}
\email{	pp396@orie.cornell.edu}

\author{J. Wang}
\address{School of Operations Research and Information Engineering, Cornell University, Ithaca, NY 14853.}
\email{	jw2333@cornell.edu}

\begin{document}

\begin{abstract}
We develop a new methodology for the fluctuation theory of  continuous-time skip-free Markov chains, extending the recent work of Choi and Patie \cite{Choi-Patie} for discrete-time skip-free Markov chains.
As the main application we use it to derive a full set of fluctuation identities regarding exiting a finite or infinite interval for Markov branching processes with immigration, thereby uncovering many new results for this classical family of continuous-time Markov chains. The theory also allows us to recover in a simple manner fluctuation identities for skip-free downward compound Poisson processes.
\end{abstract}
 \maketitle

\section {Introduction}

First passage times are ubiquitous in probability theory and, more broadly, in many fields of sciences, such as epidemiology, physics, neurology, insurance and financial mathematics. In spite of the wide applicability of first-passage phenomena, finer characterizations of the law of these random variables
beyond as solutions of a boundary value problem associated with a linear operator, have been obtained
only for very few and specific classes of Markov processes, including most notably one-dimensional diffusion processes and their discrete-state space analogue birth-death chains on the one hand and processes with stationary and independent increments on the other hand, both for which tailor-made approaches have been developed.
Unfortunately, attempts to carry over techniques from these cases have had limited success, providing some
evidence towards the need to find a novel way to tackle this issue. In this vein, in the recent paper \cite{Choi-Patie}, Choi and Patie develop an original  and  comprehensive approach based on a combination of techniques such as  the theory of Martin boundary and potential theory to provide an expression for
the resolvent of a discrete-time skip-free Markov chain in terms of the so-called fundamental excessive functions of three different chains, namely the given skip-free downward (upward) Markov chain, its dual and the original one killed when going above (below) a given level. This allowed them to derive various fluctuation identities involving the first hitting time of a given point and the first exit times out of a given finite or semi-infinite interval.

In this paper we will follow the approach of Choi and Patie \cite{Choi-Patie} to extend their results to the case of continuous-time skip-free Markov chains, while also improving on their methodology as explained in the beginning of Section \ref{sec_pot_fluc}. This is our first main contribution. Besides our time homogeneous chain being skip-free (downward) we further assume that all its states are stable (i.e.~no instantaneous states) and transient, though we allow for state-dependent killing and for explosion with the precise setting being laid out in Section \ref{sec_prelim}. Though we already mention that for our applications covered in Sections \ref{sec_CPP} and \ref{sec_MBI} we can also deal with the cases where not all states are transient by employing the usual trick of including a constant killing rate, say $q>0$, which makes the chain transient, and then letting $q\downarrow 0$.
Our second main contribution is to use the general theory presented in Section \ref{sec_pot_fluc} to derive a collection of fluctuation identities for Markov branching processes with immigration (MBIs), a well-studied class of continuous-time Markov chains. Most of these identities are new as explained at the end of Section \ref{sec_MBI}, which shows the power of our methodology. As a quick illustration of the general theory and in order to contrast it with the case of MBIs we will also derive the same type of fluctuation identities for continuous-time skip-free Markov chains with stationary and independent increments, i.e.~skip-free compound Poisson processes (CPPs), though for this class the resulting identities are (essentially) known.

Of course for deriving certain probabilities involving first passage times there is the analytical approach. To illustrate this, let $p(x)$ be the probability that our skip-free downward continuous-time Markov chain  with state space $E\subset \mathbb Z$ and starting at $x\in E$ exits a given interval $(a,b)\subset E$ by hitting $a$. Then, with some of the terminology explained later on in Section \ref{sec_prelim}, $p(x), x\in[a-1,b+1]$, is harmonic for the Markov chain killed at exiting $(a,b)$ and as a consequence, with $Q$  the Q-matrix of the chain, $p(x)$, $x\geq a$, solves $\sum_{y\in E}Q(x,y)p(y)$, $x\in[a+1,b-1]$, with the boundary conditions $p(a)=1$ and $p(x)=0$ for $x\geq b$. Consequently, $p(x)$ can be computed by solving a linear system of $b-a-1$ equations. Note that this approach becomes more difficult when $E$ is unbounded and $a=-\infty$ or $b=\infty$ as one then has to deal with an infinite system of equations. Our main theorems in Section \ref{sec_pot_fluc} provide in particular a structure for the solution of these systems. Besides yielding more insight into the solution, even in the case of a finite system, our results for MBIs and CPPs allow one to compute probabilities like $p(x)$ for these classes in a more efficient way than by numerically solving a system of linear equations, especially when one wants to compute $p(x)$ for various values of $a$ and $b$.

The rest of the paper is organized as follows. In Section 2 we introduce various notations and concepts that will be needed later on and in particular define the class of Markov chains that we will be working with. Section \ref{sec_pot_fluc} contains the main results where we, following Choi and Patie \cite{Choi-Patie}, express, for a general Markov chain, the resolvent and derive various fluctuation identities in terms of the three fundamental excessive functions, whereas in Sections \ref{sec_CPP} and \ref{sec_MBI} we apply these results to the class of skip-free CPPs and MBIs respectively.

\section{Preliminaries}
\label{sec_prelim}

Let $E=[\l,\r]\cap\mathbb Z$ where $\l,\r\in\mathbb Z\cup\{-\infty,\infty\}$ with $\l<\r$. We set $E_\partial=E\cup\{\partial\}$ where $\partial$ denotes the cemetery state. We always work in a canonical setup where (i) $\Omega$ is the space of functions $\omega:[0,\infty)\to E_\partial$ with the property that there exists $\zeta(\omega)\in[0,\infty]$ such that $\omega$ is $E$-valued and right-continuous on $[0,\zeta(\omega))$ and $\omega(t)=\partial$ for $t\geq\zeta(\omega)$, (ii) $X=(X_t)_{t\geq 0}$ is the process defined by $X_t(\omega)=\omega(t)$ for $\omega\in\Omega$, (iii) $(\mathcal{F}_t)_{t\geq 0}$ is the natural filtration of $X$ and (iv) $\mathcal F$ is the smallest sigma-algebra containing $\cup_{t\geq 0}\mathcal F_t$. The random variable $\zeta$ then represents the lifetime $\inf\{t\geq 0:X_t=\partial\}$ of $X$. With $J_0=0$ and $J_n=\inf\{t\geq J_{n-1}:X_t\neq X_{J_{n-1}}\}$ for $n\geq 1$ , the random variables $J_1,J_2,\ldots$ are the consecutive jump times of $X$. We denote by $J_\infty=\lim_{n\to\infty}J_n$ the explosion time.

Let $Q=(Q(x,y))_{x,y\in E}$ be a given $Q$-matrix on $E$, i.e.~a matrix satisfying $0\leq Q(x,y)<\infty$ for $x\neq y$ and $\sum_{y\neq x}Q(x,y)\leq Q(x)<\infty$ where $Q(x):=-Q(x,x)$. We set $Q(x,\partial)=Q(x)-\sum_{y\neq x}Q(x,y)\geq 0$ for $x\in E$. Note that $Q$ is not assumed to be conservative, i.e. $Q(x,\partial)$ is not necessarily equal to $0$.
Given $Q$ we next specify a continuous-time, time homogeneous Markov chain on $(\Omega,\mathcal F)$ via the usual jump chain/holding time definition.
Namely, we assume $\mathbb P=(\mathbb P_x)_{x\in E_\partial}$ is a family of probability measures on $(\Omega,\mathcal F)$ such that (i) $\mathbb P_x(X_0=x)=1$ for $x\in E_\partial$, (ii)  for all $n\geq 1$ and $x\in E$, given $J_{n-1}<\infty$ and $X_{J_{n-1}}=x$ the $n$-th holding time $H_n:=J_n-J_{n-1}$ is exponentially distributed with parameter $Q(x)$ if $Q(x)>0$ and equal to $\infty$ otherwise and is further conditionally independent from $\mathcal F_{J_{n-1}}$ (iii) for all $n\geq 1$, $x\in E$ and $y\in E_\partial$,  given $J_{n-1}<\infty$ and $X_{J_{n-1}}=x$, $X_{J_n}=y$ with probability $\frac{Q(x,y)}{Q(x)}$ if $Q(x)>0$ and $X_{J_n}$ is conditionally independent from $\mathcal F_{J_{n-1}}$ and $H_n$, (iv) $X$ is sent to the cemetery state if explosion occurs, i.e.~$X_{J_\infty}=\partial$ if $J_\infty<\infty$. We call $(X,\mathbb P)$ the Markov chain generated by $Q$ with state space $E$. Note that the Markov chain can end up in the cemetery state either due to explosion or due to a transition to $\partial$ by the jump chain $(X_{J_n})_{n\geq 0}$.

We write $P_t(x,y)=\mathbb P_x(X_t=y)$, $x,y\in E$, $t\geq 0$, for the transition kernel of $(X,\mathbb P)$.
From the theory of continuous time Markov chains (see e.g. Section 2.8 in \cite{norris}\footnote{In this reference it is assumed that $Q$ is conservative but the arguments go through if $Q$ is not conservative.}), the transition kernel $P_t(x,y)$, $x,y\in E$, $t\geq 0$, is the minimal nonnegative solution of both the forward equations, $\frac d{d t}P_t(x,y)=\sum_{z\in E}P_t(x,z)Q(z,y)$, $P_0(x,y)=\delta_{xy}$ with $\delta_{xy}$ the Kronecker delta, as well as the backward equations, $\frac d{d t}P_t(x,y)=\sum_{z\in E}Q(x,z)P_t(z,y)$, $P_0(x,y)=\delta_{xy}$.
We will use the notation $P_tf(x):=\sum_{y\in E} P_t(x,y)f(y)=\mathbb E_x[f(X_t)\1_{\{t<\zeta\}}]$ and $Qf(x):=\sum_{y\in E} Q(x,y)f(y)$ for a function $f:E\to\mathbb R$ and where $\mathbb E_x$ is the expectation operator associated with $\mathbb P_x$.

Throughout the paper we assume the Q-matrix $Q$ is such that for $x,y\in E$,
\begin{equation}\label{skipfree_downward_assump}
Q(x,x-1)>0 \quad \text{if $x>\l$} \quad \text{and} \quad Q(x,y)=0 \quad \text{if $y\leq x-2$}.
\end{equation}
This assumption means that $(X,\mathbb P)$ is skip-free downwards: from any state $(X,\mathbb P)$ can reach a lower state with positive probability but only by visiting all states in between. We further assume, for the rest of this section, that all states in $E$ are transient.
The resolvent of $(X,\mathbb P)$ is defined by $G(x,y)=\int_{0}^\infty P_t(x,y) d t$. By the transience assumption  $G(x,y)<\infty$ for all $x,y\in E$ (see e.g. Theorem 3.4.2 in \cite{norris}) and by the skip-free downward assumption in combination with the right-continuity of the sample paths $G(x,y)>0$ for all $x,y\in E$ such that $y\leq x$.

We say that a function $h:E\to[0,\infty)$ is excessive, respectively harmonic, for $(X,\mathbb P)$, if $P_t h(x)\leq h(x)$, respectively $P_t h(x)=h(x)$, for all $x\in E$ and $t\geq 0$. Note that by the semi-group property of $P_t(x,y)$,  $P_t h(x)$ is nonincreasing in $t$ when $h$ is excessive and so $\lim_{t\downarrow 0}P_t h(x)=h(x)$ by the monotone convergence theorem and the right-continuity of $P_t$ at $t=0$. Similarly, we call a measure $\pi:E\to[0,\infty)$  excessive, respectively stationary,  if $\sum_{x\in E}\pi(x)P_t(x,y)\leq \pi(y)$, respectively $\sum_{x\in E}\pi(x)P_t(x,y)= \pi(y)$, for all $y\in E$ and $t\geq 0$.
Under our assumptions there exists an excessive measure $\pi$ for $(X,\mathbb P)$ that is positive (i.e. $\pi(y)>0$ for all $y\in E$), e.g. $\pi(y) = \sum_{x\in E} 2^{-\sigma(x)} G(x,y)$ for some injective map $\sigma:E\to\{1,2,\ldots\}$ will do, see also Section 5.2 in \cite{KSK}. We now fix such a positive, excessive measure $\pi$ as our reference measure and we denote by $\mathtt g(x,y)=G(x,y)/\pi(y)$ the resolvent density.

Next we introduce some auxiliary Markov chains associated with $(X,\mathbb P)$.
First, writing $\mathbb{\widehat P}=(\widehat{\mathbb P}_y)_{y\in E_\partial}$, the {dual} $(X,\mathbb{\widehat P})$ of $(X,\mathbb P)$ with respect to $\pi$ is the Markov chain with state space $E$ generated by the Q-matrix $\widehat Q$ given by
\begin{equation*}
\widehat Q(y,x) = \frac{\pi(x)Q(x,y)}{\pi(y)}, \quad  y,x\in E.
\end{equation*}
Note that $\widehat Q$ is a well-defined Q-matrix since $\pi$ is a positive measure and, for each $y\in E$,
\begin{equation*}
\begin{split}
\pi(y)\sum_{x\in E}\widehat Q(y,x) = \sum_{x\in E} \pi(x)Q(x,y) \leq & \liminf_{t\downarrow 0}\sum_{x\neq y}\pi(x)\frac{P_t(x,y)}t + \pi(y)\lim_{t\downarrow 0} \frac{P_t(y,y)-1}t \\
= & \liminf_{t\downarrow 0}\frac{\sum_{x\in E} \pi(x)P_t(x,y) - \pi(y)}t \\
\leq & 0,
\end{split}
\end{equation*}
where the first inequality is due to Fatou's lemma and, from the backward equations, $\frac{d}{d t}P_t(x,y)|_{t=0}=Q(x,y)$ and the last inequality is because $\pi$ is an excessive measure for $(X,\mathbb P)$. Denoting by $P_t(y,x):=\widehat{\mathbb P}_y(X_t=x)$ the transition kernel of $(X,\mathbb{\widehat P})$, we easily see
that $\frac{\pi(x)P_t(x,y)}{\pi(y)}$, $y,x\in E$, $t\geq 0$, is the minimal nonnegative solution of the backward (or forward) equations associated with $\widehat Q$ and therefore  $\widehat P_t(y,x)=\frac{\pi(x)P_t(x,y)}{\pi(y)}$.
Denoting by $\widehat G$ and $\widehat{\mathtt g}$ respectively the resolvent and the resolvent density (with respect to $\pi$) of the dual, we immediately observe that $\widehat{\mathtt g}(y,x)=\mathtt g(x,y)$, $x,y\in E$.
Now denote by $T_A$  the first hitting time of the set $A \subseteq  \mathbb R$, that is
\[T_A = \inf\{t \geq 0; X_t \in A\},\]
with the usual convention that $\inf \{\emptyset\} = \infty$. We will write $T_a:=T_{\{a\}}$ for $a\in E$. For $A\subseteq \mathbb R$ and $B\subseteq E_\partial$ we write $B^A:=B\backslash A$. We will also use the notation $[b$ for the set $[b,\infty)$. %
For $A\subseteq \mathbb R$ and writing $\mathbb P^A=\left(\mathbb P^A_x\right)_{x\in E_\partial^A}$,
the second auxiliary Markov chain  $(X,\mathbb P^A)$  is the one with state space $E^A$ and generated by the Q-matrix  $Q^A$ given by
\begin{equation*}
Q^A(x,y) = Q(x,y), \quad x,y\in E^A.
\end{equation*}
Let $P^A_t(x,y):=\mathbb P^A_x(X_t=y)$  be the transition kernel of $(X,\mathbb P^A)$. By mimicking Steps 1 and 2 of the proof of Theorems 2.8.3 and 2.8.4 in \cite{norris} one can show that $\mathbb P_x(X_t=y,t<T^A)$, $x,y\in E^A$, $t\geq 0$, is the minimal nonnegative solution of the backward equations associated with $Q^A$ and thus $P^A_t(x,y)=\mathbb P_x(X_t=y,t<T_A)$. Therefore we call $(X,\mathbb P^A)$ the Markov chain obtained by killing $(X,\mathbb P)$ at $T_A$. Using the semi-group property one can show that, for all $t\geq 0$, $x\in E$ and $\Gamma\in\mathcal F_t$,  $\mathbb P_x^A(\Gamma)=\mathbb P_x(\Gamma\cap\{t<T_A\})$. Then, for any stopping time $T$ one can show by first considering stopping times with a discrete range and then approximating a general stopping time by a non-increasing sequence of stopping times with a discrete range,
\begin{equation}\label{killed_process_stoptime}
\mathbb P_x^A(\Gamma\cap\{T<\infty\})=\mathbb P_x (\Gamma\cap\{T<T_A\} ), \quad \Gamma\in\mathcal F_T, \ x\in E.
\end{equation}
The resolvent of $(X,\mathbb P^A)$ is denoted by $G^A(x,y)=\int_0^\infty P^A_t(x,y) d t $, $x,y\in E^A$ and we extend it to the whole of $E\times E$ by setting $G^A(x,y)=0$ for $(x,y)\notin E^A\times E^A$. The corresponding resolvent density is denoted by $\mathtt  g^A(x,y):=G^A(x,y)/\pi(y)$, $x,y\in E$.
We can consider the dual of $(X,\mathbb P^A)$ or kill the dual at $T_A$. Clearly by looking at the Q-matrices involved, the operations of taking the dual (with respect to $\pi$) and killing at $T_A$ commute and the following {switching identity} holds:
\begin{equation}\label{switching_identity}
\widehat{\mathtt g}^A(y,x) = \mathtt g^A(x,y), \quad x,y\in E,
\end{equation}
where $\widehat G^A(y,x)$ is the resolvent and $\widehat{\mathtt g}^A(y,x)$ its density of the dual killed at $T_A$.
Very closely related to $(X,\mathbb P^A)$, is the third auxiliary Markov chain that we introduce. For $A\subseteq \mathbb R$ and writing $\overline{\mathbb P}^A=\left(\overline{\mathbb P}^A_x\right)_{x\in E_\partial}$, let
$(X,\overline{\mathbb P}^A)$ be the Markov chain with state space $E$ and generated by the Q-matrix $\overline Q^A$ given by
\begin{equation*}
\overline Q^A(x,y) =
\begin{cases}
 Q(x,y) &  \text{if $x\in E^A$ and $y\in E$}, \\
 0 & \text{if $x\in A$ and $y\in E$}.
\end{cases}
\end{equation*}
Similarly as for the previous auxiliary Markov chain, one can show that the transition kernel $\overline P^A_t(x,y):=\overline{\mathbb P}^A_x(X_t=y)$ of $(X,\overline{\mathbb P}^A)$ admits the expression $\overline P^A_t(x,y) = \mathbb P_x(X_{t\wedge T_A}=y)$, $x,y\in E$, $t\geq 0$.  Therefore we call $(X,\overline{\mathbb P}^A)$ the Markov chain obtained by stopping $(X,\mathbb P)$ at $T_A$.
Finally, for $a\in E$, with slight abuse of notation, we write $E^a_\partial:=E_\partial\cap[a,\infty)$ and $\overline{\mathbb P}^a=\left(\overline{\mathbb P}^a_x\right)_{x\in E^a_\partial}$ and denote by $(X,\overline{\mathbb P}^a)$ the Markov chain obtained by {stopping} $(X,\mathbb P)$ {at $T_a$}. Note that the state space of $(X,\overline{\mathbb P}^a)$ is chosen to be $E^a:=E\cap[a,\infty)$ and not $E$, which we can do because $(X,\mathbb P)$ is skip-free downward. Its transition kernel is denoted by $\overline P^a_t$.

\section{Potential theory and fluctuation identities of continuous-time skip-free Markov Process}\label{sec_pot_fluc}

The following two theorems are, to some extent, the analogue of the results in Section 3 and the first part of Section 4 in Choi and Patie \cite{Choi-Patie} who deal with discrete time Markov chains that are skip-free upwards.
Though there are some differences with the most notable ones being the following. First, translating their approach to our setting of continuous-time skip-free downward Markov chains, Choi and Patie \cite{Choi-Patie} distinguish between the cases (i) $\l=-\infty$ or $\l>-\infty$ with $\mathbb P_\l(X_{J_1}\in E\backslash\{\l\})=0$ on the one hand and (ii) $\l>-\infty$ with $\mathbb P_\l(X_{J_1}\in E\backslash\{\l\})>0$ on the other hand, with their results being less clean in the latter case. By using a different definition for the fundamental excessive function for the killed process than the one in Theorem 3.1(2) in \cite{Choi-Patie} (though both definitions are equivalent in case (i)), we are able to unify the two cases and at the same time give a more elementary proof that avoids the Martin boundary theory employed in \cite{Choi-Patie}.
Second, we allow for state-dependent killing whereas Choi and Patie \cite{Choi-Patie} worked with a constant killing rate.
Recall that notation and terminology used in the following two theorems can be found in Section \ref{sec_prelim}.
\begin{theorem}\label{thm_fundexc}
Let $(X,\mathbb P)$ be a Markov chain with state space $E$ generated by the Q-matrix $Q$ satisfying \eqref{skipfree_downward_assump}. Assume further all states in $E$ are transient and fix a reference point $\mfo\in E$.
\begin{enumerate}[label=(\arabic*),ref=(\arabic*)]
	\item\label{item_H} Define, for $x\in E$, $H(x)=\lim_{y\downarrow\l} \frac{G(x,y)}{G(\mfo,y)}: =
	 \begin{cases}
	\lim_{y\downarrow-\infty} \frac{G(x,y)}{G(\mfo,y)} & \text{if $\l=-\infty$}, \\
	\frac{G(x,\l)}{G(\mfo,\l)} & \text{if $\l>-\infty$}.
	\end{cases}$
	Then for $x,a\in E$ such that $x\geq a$,
	\begin{equation}\label{fpt_above_general}
		\mathbb P_x(T_a<\zeta) = \frac{H(x)}{H(a)}.
	\end{equation}
	\item\label{item_charac_H} The function $H$ is excessive for $(X,\mathbb P)$ and it is the unique positive-valued function on $E$ satisfying the following four properties: (i) nonincreasing, (ii) $H(\mfo)=1$, (iii) for any $a\in E$,  $H$ restricted to $E^a$ is harmonic for $(X,\overline{\mathbb P}^a)$, (iv) if $\r=\infty$ and if there exists $x>a$ such that $\mathbb P_x(T_a=\zeta=\infty)$, then $\lim_{x\to\infty}H(x)=0$.
	\item\label{item_resolv} Define $\widehat H(y)=\lim_{x\uparrow\r} \frac{\widehat G(y,x)}{\widehat G(\mfo,x)}$, $y\in E$, and $\Hqb(x)=\lim_{y\downarrow\l} \frac{G^{[b}(x,y)}{G(\mfo,y)}$, $x\in E$. Then for $x,y\in E$,
	\begin{equation}\label{resolv_expr}
	\mathtt g(x,y)= \con \widehat H(y) \left( H(x)-\Hqy(x) \right) \quad \text{with \ $\con=\mathtt g(\mfo,\mfo)$}.
	\end{equation}
\end{enumerate}		
\end{theorem}
\begin{proof}
In the proof $x,y,z,a,b$ will always denote elements in $E$.

\ref{item_H}. We first prove the following claim:
\begin{equation}\label{martinkernel_eventuallyconstant}
\frac{G(x,y)}{G(\mfo,y)} = \frac{G(x,x\wedge \mfo)}{G(o,x\wedge \mfo)}  \quad  y\leq x\wedge \mfo.
\end{equation}
By the strong Markov property, $G(z,a)=\mathbb P_z(T_{a}<\zeta)G(a,a)$ for any $z,a\in E$.
Hence,
\begin{equation*}
\begin{split}
\frac{G(b,y)}{G(y,y)} =  \mathbb P_b(T_{y}<\zeta) = & \mathbb P_b (T_{z}<\zeta)\mathbb P_z(T_{y}<\zeta)
=   \frac{G(b,z)}{G(z,z)}\frac{G(z,y)}{G(y,y)}, \quad y\leq z\leq b,
\end{split}
\end{equation*}
where the middle equality follows by the strong Markov property and the skip-free downward property of $(X,\mathbb P)$.
When $x\leq\mfo$, take $z=x$ and $b=\mfo$, to get for $y\leq x\leq \mfo$, $G(\mfo,y) =   \frac{G(\mfo,x)}{G(x,x)} G(x,y)$ which shows \eqref{martinkernel_eventuallyconstant} in the case where  $x\leq\mfo$. Similarly, if $\mfo\leq x$, take $z=\mfo$ and $b=x$ instead to get for $y\leq \mfo\leq x$, $G(x,y)=\frac{G(x,\mfo)}{G(\mfo,\mfo)}G(\mfo,y)$, which shows \eqref{martinkernel_eventuallyconstant} in the case where $x\geq\mfo$. From \eqref{martinkernel_eventuallyconstant} it follows that $H(x)=\lim_{y\downarrow\l} \frac{G(x,y)}{G(\mfo,y)}$ is well-defined. By the strong Markov property and the skip-free downward property, $G(x,y)=\mathbb P_x(T_{a}<\zeta)G(a,y)$  for $y\leq a\leq x$. Hence
\begin{equation*}
\mathbb P_x(T_{a}<\zeta) = \frac{G(x,y)}{G(a,y)} = \frac{G(x,y)/G(\mfo,y)}{G(a,y)/G(\mfo,y)}, \quad y\leq a\leq x.
\end{equation*}
Now taking limits as $y\downarrow\l$ yields \eqref{fpt_above_general}.

\ref{item_charac_H}. By the semi-group property of $P_t$, for $y\leq x$,
\begin{equation*}
\sum_{z\in E} P_t(x,z) G(z,y) =\sum_{z\in E} P_t(x,z) \int_0^\infty P_s(z,y) d s = \int_0^\infty P_{t+s}(x,y) d s = \int_t^\infty P_u(x,y)d u < G(x,y). %
\end{equation*}
So if $\l>-\infty$, $P_t H(x)<H(x)$ for all $x\in E$ and so $H$ is excessive (but not harmonic) for $(X,\mathbb P)$. On the other hand if $\l=-\infty$, then $H$ is the limit as $n\to\infty$ of the excessive functions $H_n(x):=\frac{G(x,-n)}{G(\mfo,-n)}$, $-n\in E$, and thus is excessive since by Fatou's lemma,
\begin{equation*}
P_t H = P_t \lim H_n \leq \liminf P_t H_n \leq \liminf H_n = H.
\end{equation*}
The function $H$ is nonincreasing by \eqref{fpt_above_general} and $(X,\mathbb P)$ being skip-free downward. Obviously, $H(\mfo)=1$ by definition.
Fix $a\in E$ and let $h_a(x):=\mathbb P_x(T_a<\zeta)$, $x\in E^a$. Then for any $x\in E^a$ and $t\geq 0$,
\begin{equation*}
\begin{split}
\overline P_t^a h_a(x) = & \mathbb E_x \left[ h_a(X_{t\wedge T_a})\1_{\{t\wedge T_a<\zeta\}} \right] \\
= & \mathbb E_x \left[ \1_{\{t\geq T_a,T_a<\zeta\}}\mathbb P_{X_{T_a}}(T_a<\zeta) \right] + \mathbb E_x \left[ \1_{\{t<T_a\wedge\zeta\}}\mathbb P_{x}(T_a<\zeta|\mathcal F_t) \right] \\
= &  \mathbb P_x(t\geq T_a,T_a<\zeta) + \mathbb P_{x}(t<T_a\wedge\zeta,T_a<\zeta) \\
= & h_a(x),
\end{split}
\end{equation*}
where we have used the Markov property in the second equality. So $h_a$ is harmonic for $(X,\overline{\mathbb P}^a)$. Then by \eqref{fpt_above_general} $H$ restricted to $E^a$ is harmonic for  $(X,\overline{\mathbb P}^a)$.
Let $h:E\to(0,\infty)$ satisfy (i) and (iii). Then for $x\in E^a$,
\begin{equation}\label{uniqueness_eq}
h(x) {=} \lim_{t\to\infty} \overline P_t^a h(x)  {=}  \mathbb E_x \left[ \lim_{t\to\infty} h(X_{t\wedge T_a})\1_{\{t\wedge T_a<\zeta\}} \right] = h(a)\mathbb P_x(T_a<\zeta) + \lim_{y\to\infty}h(y) \mathbb P_x(T_a=\zeta=\infty) ,
\end{equation}
where the second equality is due to the dominated convergence theorem which is applicable because $h$ is nonincreasing and positive, $x\geq a$ and $(X,\mathbb P)$ is skip-free downward and for the last equality we used that by transience $X_t\to\infty$ as $t\to\infty$ if $T_a=\zeta=\infty$ (which can only happen if $\r=\infty$). So if $h$ satisfies (i)-(iv), then we conclude $h(x) = h(a)\mathbb P_x(T_a<\zeta)$ for all $x\geq a$, which yields that $h$ is proportional to $H$ on $E^a$ by \eqref{fpt_above_general}.  Since $a\in E$ was chosen arbitrarily, $h$ is positive-valued and $h(\mfo)=H(\mfo)$, it follows that  $h=H$ on $E$ if $h$ satisfies (i)-(iv). Since we have showed that $H$ satisfy (i) and (iii), \eqref{uniqueness_eq} holds for $H$ and thus $H$ satisfies property (iv) by \eqref{fpt_above_general}.

\ref{item_resolv}. By the strong Markov property and the skip-free upward property of the dual, for $x,y,b\in E$ such that $y\leq b$,
\begin{equation*}
\widehat G(y,x) = \widehat G^{[b}(y,x) + \widehat{\mathbb P}_y (T_{[b}<\zeta) \widehat G(b,x)
= \widehat G^{[b}(y,x) + \frac{\widehat H(y)}{\widehat H(b)} \widehat G(b,x),
\end{equation*}
where for the second equality we used the analogue of \eqref{fpt_above_general} for the dual which is justified because $(X,\widehat{\mathbb P})$ is skip-free upward and all its states are transient. Then  by the switching identity \eqref{switching_identity},
\begin{equation}\label{killedresolvent}
G^{[b}(x,y) %
= G(x,y) -   G(x,b)\frac{\widehat H(y)\pi(y)}{\widehat H(b)\pi(b)}, \quad y\leq b.
\end{equation}
Since by the same arguments $\frac{G(\mfo,y)}{\pi(y)}=\mathtt g(\mfo,y)=\widehat{\mathtt g}(y,\mfo)=\widehat{\mathbb P}_y(T_\mfo\leq\zeta)\mathtt g(\mfo,\mfo)=\con\frac{\widehat H(y)}{\widehat H(\mfo)}$ for $y\leq \mfo$,
\begin{equation*}
\frac{G^{[b}(x,y)}{G(\mfo,y)} %
= \frac{G(x,y)}{G(\mfo,y)} -   \frac{G(x,b)}{\con\widehat H(b)\pi(b)}, \quad  y\leq b\wedge \mfo.
\end{equation*}	
Taking limits as $y\downarrow -\l$ shows that $\Hqb$ is well-defined and $\Hqb(x)=H(x)-\frac{G(x,b)}{\con\widehat H(b)\pi(b)}$. The expression for the resolvent density of $G$ then follows.
\end{proof}

\begin{remark}\label{remark_harm}
	\begin{enumerate}[leftmargin=2em]
		\item The function $H$ is not necessarily harmonic for $(X,\mathbb P)$ if $\l=-\infty$, which is in contrast to the discrete time case, see \cite{Choi-Patie}. For instance when $(X,\mathbb P)$ is the negative of a birth process which explodes, then $H\equiv 1$ but $H$ is not harmonic since the lifetime is finite a.s..
		
		\item If $h:E\to(0,\infty)$ is harmonic for $(X,\mathbb P)$ then $h$ satisfies property (iii) in Theorem \ref{thm_fundexc}\ref{item_charac_H}. Indeed, if $h$ is harmonic, then by the Markov property for any $s\leq t$ and $x\in E$, $\mathbb E_x[h(X_t)\1_{\{t<\zeta\}}|\mathcal F_s]=\1_{\{s<\zeta\}} \mathbb E_{X_s}[h(X_{t-s})\1_{\{t-s<\zeta\}}]=h(X_s)\1_{\{s<\zeta\}}$. So $(h(X_t)\1_{\{t<\zeta\}})_{t\geq 0}$ is a martingale and therefore so is $(h(X_{t\wedge T_a})\1_{\{{t\wedge T_a}<\zeta\}})_{t\geq 0}$ for any $a\in E$. This yields, for any $a\in E$, $x\in E^a$ and $t\geq 0$, $\overline P_t^a h(x)=\overline P_0^a h(x)=h(x)$.
		
		\item Property (iv) in Theorem \ref{thm_fundexc}\ref{item_charac_H} is needed to characterise $H$. Indeed, if $\r=\infty$ and $(X,\mathbb P)$ is such that $\lim_{t\to\infty}X_t=\infty$ $\mathbb P_x$-a.s., then $h\equiv 1$ satisfies properties (i)-(iii) but $H\neq h$.
	\end{enumerate}
\end{remark}

\begin{theorem}\label{thm_fluc}
	Under the assumptions of Theorem \ref{thm_fundexc} and using its notation we have the following.
\begin{enumerate}[label=(\arabic*),ref=(\arabic*)]
	\item\label{item_fht} For any $x,y\in E$,
	\begin{equation}\label{fht_expr}
	\mathbb P_x(T_y<\zeta) = \frac{H(x)-\Hqy(x)}{H(y)}.
	\end{equation}

	\item\label{item_twosidedexit} For any $x,a,b\in E$ such that $x\geq a$ and $b>a$,
	\begin{equation}\label{twosidedexit}
	\mathbb P_x(T_a<T_{[b}\wedge\zeta) = \frac{\Hqb(x)}{\Hqb(a)}.
	\end{equation}

	\item\label{item_dynkin} Let $a,b\in E$ such that $a<b$. Assume $f:E\to\mathbb R$ is such that $Q|f|(x)<\infty$ for all $x\in E^{(a,b)^\complement}$ where $|f|(x):=|f(x)|$. Then for $x\in E$,
	\begin{equation}\label{dynkin}
	\mathbb E_x \left[ f \left( X_{T_{(a,b)^\complement}} \right) \1_{\{ T_{(a,b)^\complement} <\zeta \}} \right]
	= f(x) + \sum_{z\in E^{(a,b)^\complement}} Qf(z) G^{{(a,b)^\complement}}(x,z),
	\end{equation}
	where the density of $G^{{(a,b)^\complement}}$
	 (with respect to $\pi$) admits the expression,
	\begin{equation}\label{2sidedkilledresol_expr}
	\mathtt g^{{(a,b)^\complement}}(x,y) = \con  \widehat H(y) \left( \frac{ \Hqb(x)}{\Hqb(a)}\Hqy(a) -\Hqy(x) \right) , \quad x,y\in E^{(a,b)^\complement}.
	\end{equation}
	In particular for $x\in E$,
	\begin{equation}\label{exit_interval}
	\mathbb P_x (T_{(a,b)^\complement} <\zeta)
	 =	1 - \sum_{z\in E^{(a,b)^\complement}} Q(z,\partial) G^{{(a,b)^\complement}}(x,z).
	\end{equation}
	
	\item\label{item_fpt_downward} If $\l=-\infty$, denote $\Hqb(-\infty):=\lim_{x\to-\infty}\Hqb(x)\in(0,\infty]$. Then for $x,b\in E$,
	\begin{equation}\label{fpt_downward}
	\mathbb P_x(T_{[b}<\zeta) =
	\begin{cases}
	1 - \sum_{z\in E^{[b}} Q(z,\partial) G^{[b}(x,z) - \frac{\Hqb(x)}{\Hqb(-\infty)} & \text{if $\l=-\infty$}, \\
	1 - \sum_{z\in E^{[b}} Q(z,\partial) G^{[b}(x,z)  & \text{if $\l>-\infty$},
	\end{cases}
	\end{equation}
	where the density of $G^{[b}$ admits the expression,
	\begin{equation}\label{1sidedkilledresol_expr}
	\mathtt g^{[b}(x,y) = \con \widehat H(y) \left( \Hqb(x) - \Hqy(x) \right)  , \quad x,y\in E^{[b}.
	\end{equation}
	
\end{enumerate}
\end{theorem}
\begin{remark}
	Although somewhat hidden, the identities in Theorem \ref{thm_fluc} provide in particular expressions for the Laplace transform of the various first hitting/passage times by including a constant term in the killing rates $Q(x,\partial)$, $x\in E$. Namely, if, for $q>0$, $(X,\mathbb P^{(q)})$ is the Markov chain with state space $E$ generated by the Q-matrix $Q^{(q)}$ given by $Q^{(q)}(x,y)=Q(x,y)- q \delta_{xy}$, then its transition kernel is given by $P_t^{(q)}(x,y)=e^{-qt}P_t(x,y)$ and so, following the arguments
	leading up to \eqref{killed_process_stoptime}, one has for any stopping time $T$,
	\begin{equation*}
	\mathbb P^{(q)}_x(\Gamma\cap \{T<\infty\})=\mathbb E_x \left[ e^{-q T} \1_{\Gamma\cap \{T<\infty\}} \right], \quad \Gamma\in\mathcal F_T, \ x\in E.
	\end{equation*}
	So using the identities in Theorem \ref{thm_fluc} for $(X,\mathbb P^{(q)})$ yields identities involving the Laplace transform of the corresponding first hitting/passage time with Laplace parameter $q$ for $(X,\mathbb P)$, see also Remark \ref{remark_MBI_killingparam} below.
\end{remark}
\begin{proof}
\ref{item_fht}. By the strong Markov property $G(x,y)=\mathbb P_x(T_{y}<\zeta)G(y,y)$ for $x,y\in E$. The identity \eqref{fht_expr} then follows from \eqref{resolv_expr} since $\Hqy(x)=0$ for $x\geq y$.

\ref{item_twosidedexit}. Since $(X,\mathbb P^{[b})$ is skip-free downward with all states transient \eqref{fpt_above_general} is applicable for this Markov chain. Hence by \eqref{killed_process_stoptime} and with $H^{[b}(x):=\lim_{y\downarrow\l} \frac{G^{[b}(x,y)}{G^{[b}(\mfo^{[b},y)}$ where $\mfo^{[b}\in E^{[b}$ is a reference point,
\begin{equation*}
\mathbb P_x(T_a<T_{[b}\wedge\zeta) = \mathbb P^{[b}(T_a<\zeta) = \frac{H^{[b}(x)}{H^{[b}(a)}, \quad x\geq a,
\end{equation*}
The required identity then follows since $\Hqb(x)=\lim_{y\downarrow\l} \frac{G^{[b}(x,y)}{G^{[b}(\mfo^{[b},y)} \frac{G^{[b}(\mfo^{[b},y)}{G(\mfo,y)} = H^{[b}(x)\Hqb(o^{[b})$, $x\in E$.

\ref{item_dynkin}. We first prove \eqref{2sidedkilledresol_expr} and \eqref{1sidedkilledresol_expr}. The latter follows directly from \eqref{killedresolvent} and \eqref{resolv_expr}. Since by the strong Markov property and $(X,\mathbb P)$ being skip-free downward,
\begin{equation*}
\mathtt g^{[b}(x,y) = \mathtt g^{{(a,b)^\complement}}(x,y) +  \mathbb P_x^{[b}(T_a<\zeta)\mathtt g^{[b}(a,y), \quad x\geq a, y\in E,
\end{equation*}
\eqref{2sidedkilledresol_expr} follows from \eqref{twosidedexit} and \eqref{1sidedkilledresol_expr}. Now we prove \eqref{dynkin}. Note that without loss of generality we can assume that $f$ is bounded because the general case  can be proved by approximation from the case where $f$ is bounded.
By the forward equations, $\overline P^{(a,b)^\complement}_t(x,y) = \delta_{xy} + \int_0^t \sum_{z\in E}\overline  P^{(a,b)^\complement}_s(x,z) \overline Q^{(a,b)^\complement}(z,y) d s$, $x,y\in E$, $t\geq 0$. Therefore, for $f:E\to\mathbb R$ bounded, $x\in E$ and $t\geq 0$, by definition of $\overline Q^{(a,b)^\complement}$,
 \begin{equation*}
 \begin{split}
 \left( \overline P^{(a,b)^\complement}_t f \right) (x) = & f(x) + \sum_{z\in E^{(a,b)^\complement}} Q f(z) \int_0^t \overline  P^{(a,b)^\complement}_s(x,z) d s,
 \end{split}
 \end{equation*}
 where the interchanging of the two sums is justified because $E^{(a,b)^\complement}$ is a finite set.
 By the transience assumption, $T_{(a,b)^\complement}\wedge\zeta<\infty$ $\mathbb P_x$-a.s. for all $x\in  E^{(a,b)^\complement}$. As $f$ is assumed to be bounded, the dominated convergence theorem yields,
 \begin{equation*}
 \begin{split}
 \mathbb E_x \left[ f(X_{T_{(a,b)^\complement}})\mathbf 1_{\{T_{(a,b)^\complement}<\zeta\}} \right]
 = & \lim_{t\to\infty} \mathbb E_x \left[ f(X_{t\wedge T_{(a,b)^\complement}})\mathbf 1_{\{t\wedge T_{(a,b)^\complement}<\zeta\}} \right] \\
 = &  \lim_{t\to\infty} \left( \overline P^{(a,b)^\complement}_t f \right) (x) \\
 = & f(x) + \sum_{z\in E^{(a,b)^\complement}} Q f(z) \int_0^\infty \overline  P^{(a,b)^\complement}_s(x,z) d s \\
 = & f(x) + \sum_{z\in E^{(a,b)^\complement}} Q f(z) G^{(a,b)^\complement}(x,z).
 \end{split}
 \end{equation*}
The identity \eqref{exit_interval} then simply follows by taking $f\equiv 1$ and the definition of $Q(x,\partial)$.

\ref{item_fpt_downward}. When $\l>-\infty$ we can use the same arguments for  \eqref{exit_interval} to prove \eqref{fpt_downward} since $E^{[b}$ is a finite set and $T_{[b}\wedge\zeta<\infty$ $\mathbb P_x$-a.s. for all $x\in  E^{[b}$ when  $\l>-\infty$. Now assume $\l=-\infty$. By \eqref{exit_interval} and \eqref{twosidedexit}, for $x,a,b\in E$ such that $a<b,x$,
\begin{equation*}
\mathbb P_x(T_{[b}< T_a\wedge\zeta) = \mathbb P_x (T_{(a,b)^\complement} <\zeta) - \mathbb P_x(T_a<T_{[b}\wedge\zeta)
= 1  - \sum_{z\in E^{(a,b)^\complement}} Q(z,\partial) G^{{(a,b)^\complement}}(x,z) - \frac{\Hqb(x)}{\Hqb(a)} .
\end{equation*}
Since $\l=-\infty$ we have for any $x\in E$, $\mathbb P_x$-a.s., $T_a\wedge\zeta\to \zeta$ as $a\to-\infty$ since all states are assumed to be transient and the Markov chain is sent to the cemetery state when explosion occurs. So \eqref{fpt_downward} for $\l=-\infty$ follows by taking $a\to-\infty$ in the last equation and using the monotone convergence theorem.
\end{proof}

\section{Downward skip-free compound Poisson processes}\label{sec_CPP}
In this section, we apply the methodology described in Section \ref{sec_pot_fluc} to recover various fluctuation identities for downward skip-free $\mathbb Z$-valued compound Poisson processes.  We recall that for the continuous-state space analogue of this class of Markov processes, namely spectrally positive L\'evy processes, this problem has been well studied and has found an impressive range of applications, such as insurance mathematics, epidemiology, financial mathematics and queuing theory, see \cite{Kyprianou, Doney}. For the discrete-time analogue, i.e.~for random walks, Spitzer \cite{Spitzer} solved this problem by means of the celebrated Wiener-Hopf factorization and alternative interesting proofs,  based either on excursion theory or martingales devices have been proposed, see again \cite{Kyprianou, Doney} and the references therein. However, all of these approaches rely on the stationarity and independent increments property of the random walks/L\'evy processes, which makes it difficult to extend them to a wider context. Note that the continuous-time, discrete-state space case considered here has been studied in \cite{Vidmar_levychain}.

In this section we assume that $E=\mathbb Z$  and the Q-matrix generating $(X,\mathbb P)$ is given by, for $x,y\in \mathbb Z$,
\begin{equation}\label{Qmat_CPP}
Q(x,y)=
\begin{cases}
\mu(y-x+1) \alpha   & \text{if $y\geq x+1$ or $y=x-1$}, \\
-(\alpha+p)  & \text{if $y=x$}, \\
0 & \text{if $y\leq x-2$},
\end{cases}
\end{equation}
where $\alpha>0$, $\mu=(\mu(j))_{j\geq 0}$ is a probability measure with $\mu(0)>0$ and $\mu(1)=0$ and $p\geq 0$. Note that $Q(x,\partial)=p$ for $x\in\mathbb Z$. This means that the holding times are i.i.d.~and exponentially distributed and the jump sizes are i.i.d.~and independent from the holding times. So $(X,\mathbb P)$ corresponds to a $\mathbb Z$-valued skip-free downward compound Poisson process with intensity $\alpha$, jump distribution $\mathbb P_0(X_{J_1}=j)=\mu(j+1)$, $j\geq -1$, and killed at the constant rate $p$. For $f:E\to\mathbb R$ we denote its generating function (GF) by $f^*[s]:=\sum_{y\in E}s^y f(y)$. We have for $x\in\mathbb Z$ and $s\in(0,1]$,
\begin{equation}\label{GF_Qmat_CPP}
Q(x,\cdot)^*[s] = \psi(s) s^{x-1},
\end{equation}
where, setting $\mu(j)=0$ for $j\leq -1$,
\begin{equation*}
\psi(s) = \alpha \left( \mu^*[s] -s \right) - p s.
\end{equation*}
Since $\psi$ is convex,  $\psi(0)=\alpha\mu_0>0$ and $\psi(1)=-p\leq 0$, we have $s_0 \in (0,1]$ and $\psi(s)>0$ for $s\in[0,s_0)$. Note that $P_t(x,\cdot)^*[s]$ is well-defined for $s=e^{i u}$ for $u\in\mathbb R$. Using the forward equations, Fubini and \eqref{GF_Qmat_CPP}, we get for $t\geq 0$, $x\in\mathbb Z$ and $s=e^{i u}$ with $u\in\mathbb R$,
\begin{equation*}
P_t(x,\cdot)^*[s] = 1 + \int_0^t \sum_{z\in\mathbb Z}P_v(x,z) Q(z,\cdot)^*[s] d v
= 1 + \frac{\psi(s)}s \int_0^t P_v(x,\cdot)^*[s] d v.
\end{equation*}
So by taking derivatives in $t$, we get $\frac d{d t} P_t(x,\cdot)^*[s]=\frac{\psi(s)}s P_t(x,\cdot)^*[s]$. Solving this simple ODE with the boundary condition $P_0(x,\cdot)^*[s]=s^x$ yields, for $t\geq 0$, $x\in\mathbb Z$ and $s=e^{i u}$ with $u\in\mathbb R$,
\begin{equation}\label{GF_CPP}
\mathbb E_x \left[ s^{X_t} \1_{\{t<\zeta\}} \right] = P_t(x,\cdot)^*[s] = s^x \exp \left( \frac{\psi(s)}s t \right).
\end{equation}
By analytic extension \eqref{GF_CPP} also holds for $s\in(0,1]$.
Also \eqref{GF_CPP} implies the spatial homogeneity property $P_t(x,y)=P_t(z,y-x+z)$ for $x,y,z\in\mathbb Z$. Clearly all states are transient if $p>0$. Further, as can be deduced from the theory of random walks, if $p=0$ then all states are transient if and only if $\mathbb E_0[X_1]\neq 0$ and by differentiating in $s$ both sides of \eqref{GF_CPP} we see that $\mathbb E_0[X_1]=\psi'(1)=\alpha(m-1)$ where $m:=\sum_{j\geq 2}j\mu(j)$. Hence all states are transient for $(X,\mathbb P)$ if and only if $p>0$ or $m\neq 1$. Otherwise all states are recurrent.
In order to present the fluctuation identities we introduce the function $\scale:\mathbb Z\to[0,\infty)$ defined by $\scale(x)=0$ for $x<0$, $\scale(0)=\frac1{\alpha\mu(0)}$ and for $x\geq 0$,
\begin{equation}\label{def_scale_CPP}
 \scale(x+1) = \frac1{\alpha\mu(0)} +  \frac1{\mu(0)}\sum_{j=0}^x \scale(j) \left( \overline \mu(x-j+1) + \frac p \alpha \right),
\end{equation}
where  $\overline \mu(k)=1-\sum_{j\leq k}\mu(j)$. It is easy to verify that the function $\scale$ on $\mathbb Z\cap[0,\infty)$ is identical to the one defined in Lemma \ref{lem_recursions} below and so from that lemma we get that the GF of $\scale$ equals $\scale^*[s]=\frac1{\psi(s)}$ for $s\in[0,s_0)$.
We now provide a suite of fluctuation identities for skip-free downward compound Poisson processes (killed at a constant rate). Recall that some of these identities can be found in \cite{Vidmar_levychain} whereas all of them for the closely related class of spectrally negative L\'evy processes can be found in Chapter 8 of \cite{Kyprianou}.

\begin{theorem}\label{thm_CPP}
	We have the following.
	\begin{enumerate}[label=(\arabic*),ref=(\arabic*)]
		\item\label{item_CPP_resolv} If $p>0$ or $m\neq 1$, then the resolvent of $(X,\mathbb P)$ is  given by, for $x,y\in\mathbb Z$,
		\begin{equation*}
		\begin{split}
		G(x,y) %
		=  -\frac{s_0^{x-y}}{\psi'(s_0)} - \scale(y-x-1).
		\end{split}
		\end{equation*}

		\item\label{item_CPP_htp} In any case, for $x,y\in\mathbb Z$,
		\begin{equation*}
		\mathbb P_x(T_y<\zeta) =	s_0^{x-y} +\psi'(s_0)\scale(y-x-1).
		\end{equation*}
				
		\item\label{item_CPP_fluc} In any case, for $a\leq x\leq b-1$,
		\begin{equation*}
		\begin{split}
		\mathbb P_x(T_{a}<T_{[b}\wedge\zeta) =& \frac{\scale(b-x-1)}{\scale(b-a-1)}, \\
		\mathbb P_x (T_{(a,b)^\complement} <\zeta)
		= &
		1 +  p\sum_{z=0}^{b-x-2}\scale(z) -  \frac{\scale(b-x-1)}{\scale(b-a-1)}  p  \sum_{z=0}^{b-a-2} \scale(z) , \\
		\mathbb P_x (T_{[b} <\zeta) %
		= &
		\begin{cases}
		1 +  p\sum_{z=0}^{b-x-2}\scale(z) - \frac{p s_0}{1-s_0} \scale(b-x-1) & \text{if $p>0$ or $m>1$}, \\
		1   + \psi'(1) \scale(b-x-1)   & \text{if $p=0$ and $m\leq 1$}.
		\end{cases}
		\end{split}
		\end{equation*}
	
	\end{enumerate}		
\end{theorem}
\begin{proof}
We first assume $p>0$. Then all states are transient and so we can use the results from Section \ref{sec_pot_fluc}. We choose for the reference point $\mfo=0$ and we choose for the reference measure $\pi\equiv 1$. Note that by spatial homogeneity $\sum_{y\in\mathbb Z} P_t(x,y) = \sum_{y\in\mathbb Z} P_t(0,y-x)\leq 1$ for any $x\in\mathbb Z$ and so $\pi\equiv 1$ is excessive for $(X,\mathbb P)$. We have by \eqref{GF_CPP}, for any $x\in\mathbb Z$,
\begin{equation*}
\sum_{y\in E} P_t(x,y)s_0^y = s_0^x \exp \left( \frac{\psi(s_0)}s t \right) =  s_0^x.
\end{equation*}
So $x\mapsto s_0^x$ is harmonic for $(X,\mathbb P)$ and so, via Remark \eqref{remark_harm} and since $s_0\in(0,1)$ as $p>0$, satisfies properties (i)-(iv) in Theorem \ref{thm_fundexc}\ref{item_charac_H}. Hence $H(x)=s_0^x$.
Further, as $\widehat G(y,x)=G(x,y)=G(-y,-x)$ for all $y\in\mathbb Z$, we have by the definitions of $H$ and $\widehat H$ in Theorem \ref{thm_fundexc}, for all $y\in\mathbb Z$,
\begin{equation*}
\widehat H(y) = \lim_{x\uparrow\infty} \frac{\widehat G(y,x)}{\widehat G(0,x)} =  \lim_{x\uparrow\infty} \frac{G(-y,-x)}{G(0,-x)} = \lim_{z\downarrow -\infty} \frac{G(-y,z)}{G(0,z)} = H(-y) = s_0^{-y}.
\end{equation*}
Since  $\psi(s)<0$ for $s\in(s_0,1]$ we have by \eqref{GF_CPP}, for $s\in(s_0,1]$,
\begin{equation*}
G(0,\cdot)^*[s] = \int_0^\infty  \exp \left( \frac{\psi(s)}s t \right) d t = -\frac s{\psi(s)}.
\end{equation*}
By \eqref{resolv_expr} and recalling $\Hqy(0)=0$ for $y\leq 0$, we have for $s\in(s_0,1]$,
\begin{equation}\label{GF_resol_CPP_pospart}
\begin{split}
\sum_{y\geq 1}s^y G(0,y) = G(0,\cdot)^*[s] -  \sum_{y\leq 0} \con(s/s_0)^y
=    -\frac s{\psi(s)} - \frac{\con }{1-s_0/s}
= & -\frac{(s-s_0)s + \con\psi(s)s}{\psi(s)(s-s_0)}.
\end{split}
\end{equation}
Since $G(0,y)\leq 1/p$ for all $y\in\mathbb Z$, we must have $\sum_{y\geq 1}s^y G(0,y)<\infty$ for all $s\in[0,1]$. Hence by continuity of a power series
\begin{equation*}
-\lim_{s\downarrow s_0}\frac{(s-s_0)s + \con \psi(s)s}{\psi(s)(s-s_0)} = \lim_{s\downarrow s_0}\sum_{y\geq 1}s^y G(0,y) \in [0,\infty).
\end{equation*}
But by L'H\^opital, if $\con\neq -\frac1{\psi'(s_0)}$, then
\begin{equation*}
\left| \lim_{s\downarrow s_0}\frac{(s-s_0)s + \con\psi(s)s}{\psi(s)(s-s_0)} \right| =
\left| \lim_{s\downarrow s_0} \frac{s-s_0 + s(1+\con\psi'(s)) + \con\psi(s)}{\psi'(s)(s-s_0) + \psi(s)} \right| = \infty.
\end{equation*}
So we must have $\con= -\frac1{\psi'(s_0)}$. Then by analytic continuation, \eqref{GF_resol_CPP_pospart} holds for all $s\in[0,1]$. Therefore, since $G(0,y) = \con s_0^{-y}(1 - \Hqy(0))$ for $y\in\mathbb Z$ by \eqref{resolv_expr}, we have for  $s\in[0,s_0)$,
\begin{equation*}
\sum_{y\geq 1} (s/s_0)^y \Hqy(0) = - \frac1\con \sum_{y\geq 1}s^y G(0,y) + \frac{s/s_0}{1-s/s_0} = -\psi'(s_0)\frac{s}{\psi(s)}.
\end{equation*}
So $\Hqy(0) = -\psi'(s_0)s_0^{y} \scale(y-1)$ (recall $\scale^*[s]=\frac1{\psi(s)}$ for $s\in[0,s_0)$) and so $G(0,y)=\con s_0^{-y}(1 - \Hqy(0))=-\frac{s_0^{-y}}{\psi'(s_0)} - \scale(y-1)$. By spatial homogeneity $G(x,y)=G(0,y-x)$ which proves item \ref{item_CPP_resolv} for the case $p>0$ and further implies by \eqref{resolv_expr}, for $x,y\in\mathbb Z$,
\begin{equation*}
\Hqy(x)=-\psi'(s_0)s_0^y \scale(y-x-1).
\end{equation*}

If $p=0$, we consider for $q>0$ the process $(X,\mathbb P^{(q)})$ which is defined to be the process $(X,\mathbb P)$ but killed at rate $q$, i.e. its Q-matrix is given by \eqref{Qmat_CPP} but with $p=q$. Then clearly from \eqref{GF_CPP}, the transition kernel of $(X,\mathbb P^{(q)})$ is given by $P^{(q)}_t(x,y)=e^{-q t} P_t(x,y)$.
By the same reasoning as for \eqref{killed_process_stoptime}, we further have for any stopping time $T$,
\begin{equation}\label{feynmankac_stoptime}
\mathbb P_x^{(q)}(\Gamma\cap\{T<\infty\})=\mathbb E_x \left[ e^{-q T} \1_{\Gamma\cap\{T<\infty\}}   \right] , \quad \Gamma\in\mathcal F_T, \ x\in E.
\end{equation}
Since all states are transient for $(X,\mathbb P^{(q)})$ we have with the obvious notation,
\begin{equation*}
G^{(q)}(x,y) = \int_0^\infty e^{-q t}P_t(x,y)d y=   -\frac{s_0(q)^{x-y}}{\psi'(s_0(q))-q} - \scale^{(q)}(y-x-1), \quad x,y\in\mathbb Z.
\end{equation*}
By the monotone convergence theorem $G^{(q)}(x,y)\to G(x,y)$ as $q\downarrow 0$ and it is easy to see by convexity of $\psi$ that $s_0(q)\to s_0$ and thus also $\psi'(s_0(q))\to\psi'(s_0)$ as $q\downarrow 0$. Further, via \eqref{def_scale_CPP} it is easily seen that  $\scale^{(q)}(x)\to \scale(x)$ as $q\downarrow 0$ for any $x\in\mathbb Z$. Since $\psi'(s_0)\neq 0$ when $p=0$ and $m\neq 1$, we can conclude that item \ref{item_CPP_resolv} also holds when $m\neq 1$. Note that the above expressions found for $H(x)$, $\widehat H(y)$ (given the same reference measure $\pi\equiv 1$) and $\Hqy(x)$ still hold for the case $p=0$ and $m\neq 1$ (for which we know all states are transient as well). Then the identities in items \ref{item_CPP_htp} and \ref{item_CPP_fluc} in the case $p>0$ or $m\neq 1$ immediately follow from Theorem \ref{thm_fluc} in combination with, relevant for the last identity in item \ref{item_CPP_fluc},
\begin{equation*}
\lim_{x\to\infty} \scale(x) =
\begin{cases}
\infty & \text{if $p>0$ or $m>1$}, \\
-\frac1{ \psi'(1)}  & \text{if $p=0$ and $m< 1$}.
\end{cases}
\end{equation*}
In order to see this, note that $\lim_{x\to\infty} \scale(x) = \lim_{s\uparrow 1} \Delta \scale^*[s]$ where $\Delta W(x)=W(x)-W(x-1)$, $x\in\mathbb Z$, and that (i) $s_0<1$ if $p>0$ or $m>1$ and so $\lim_{s\uparrow 1} \Delta \scale^*[s]\geq \lim_{s\uparrow s_0} \Delta \scale^*[s] = \lim_{s\uparrow s_0}\frac{1-s}{\psi(s)} =\infty$ and (ii) $s_0=1$ and $\psi(1)=0$ if $p=0$ and $m<1$ and so $\lim_{s\uparrow 1} \Delta \scale^*[s] = \lim_{s\uparrow 1} \frac{1-s}{\psi(s)} = -\frac1{\psi'(1)}$.
If on the other hand, $p=0$ and $m=1$, then $s_0=1$, $\psi'(1)=0$ and all states are recurrent. Hence item \ref{item_CPP_htp} and the last identity in item \ref{item_CPP_fluc} still hold when $p=0$ and $m=1$. The other two identities in item \ref{item_CPP_fluc}, in the case where $p=0$ and $m=1$, can be proved by considering the process $(X,\mathbb P^{(q)})$ for $q>0$ and taking limits as $q\downarrow 0$ while using \eqref{feynmankac_stoptime}.
\end{proof}

\section{Markov Branching Processes with Immigration}\label{sec_MBI}
In this section, we apply the methodology developed in Section \ref{sec_pot_fluc} to establish the fluctuation theory of Markov branching processes with immigration, which are downward skip-free continuous-time Markov chains. We emphasize that although this class of processes have been intensively studied, only the law of the downward first passage time, that is the one with continuous crossing, has been characterized through its Laplace transform, see \cite{paper1,Vidmar} and in the continuous state space setting \cite{DFM}. Further remarks about the literature can be found near the end of the section.

We assume that $E=\mathbb Z\cap[0,\infty)$ is the set of nonnegative integers and the Q-matrix generating $(X,\mathbb P)$ is given by, for $x,y\in E$,
\begin{equation}\label{QmatMBI}
Q(x,y)=
\begin{cases}
\mu(y-x+1) \alpha x + \nu(y-x) \beta    & \text{if $y\geq x+1$}, \\
-(\alpha+p) x -(\beta+q) & \text{if $y=x$}, \\
\mu(0) \alpha x   & \text{if $y=x-1$}, \\
0 & \text{if $y\leq x-2$},
\end{cases}
\end{equation}
where $\alpha>0$, $\mu=(\mu(j))_{j\geq 0}$ is a probability measure with $\mu(0)>0$ and $\mu(1)=0$, $p\geq 0$, $\beta\geq 0$, $\nu=(\nu(j))_{j\geq 1}$ is a probability measure and  $q\geq 0$. Note that $Q(x,\partial)=p x + q$ for $x\in E$.
We call $(X,\mathbb P)$ a Markov branching process with immigration (MBI) with parameters $(\alpha,\mu,p,\beta,\nu,q)$. If $\beta=q=0$ then we call the MBI a Markov branching process (MBP) with parameters $(\alpha,\beta,p)$.
MBIs are skip-free downward (note $\alpha\mu(0)>0$) continuous-time Markov chains with transition rates that are affine in the state variable and were introduced by Sevast'yanov \cite{sevastyanov}. For more background information on MBIs/MBPs we refer to Section 1 of Li et al. \cite{li_chen_pakes_2012} and the references therein or Chapter V of Harris \cite{harris} or Chapter III of Athreya and Ney \cite{Ney-Bran}.
For $f:E\to\mathbb R$ we denote its generating function (GF) by $f^*[s]:=\sum_{y\in E}s^y f(y)$ and for $f,g:E\to\mathbb R$ we denote by $f\star g:E\to\mathbb R$ the convolution of $f$ and $g$, i.e. $(f\star g)(x)=\sum_{y=0}^xf(x-y)g(y)$. The {branching} and {immigration mechanisms} $\psi$ and $\phi$ of an MBI are defined by, for $s\in[0,1]$,
\begin{equation*}%
\begin{split}
\psi(s) = & \alpha \left( \mu^*[s] -s \right) - p s,  \\
\phi(s) = & \beta \left( 1-\nu^*[s] \right) +q ,
\end{split}
\end{equation*}
where we have set $\nu(0)=0$.
Further, we define
\begin{equation*}
s_0 =\inf\{s>0:\psi(s)=0\}.
\end{equation*}	
Since $\psi$ is convex,  $\psi(0)=\alpha\mu_0>0$ and $\psi(1)=-p\leq 0$, we have $s_0 \in (0,1]$ and $\psi(s)>0$ for $s\in[0,s_0)$. From \eqref{QmatMBI},
\begin{equation}\label{conn_Qmat_mech}
 Q(x,\cdot)^*[s] = \psi(s)x s^{x-1} - \phi(s) s^x, \quad x\geq 0, \ s\in[0,1],
\end{equation}
where we understand that $x s^{x-1}=0$ when $x=0$ and $s=0$.
Before we derive the fluctuation identities for the MBI using the theory from Section \ref{sec_pot_fluc}, we need a few preparatory lemmas.
\begin{lemma}\label{lem_recursions}
	Let $\overline\mu=(\overline\mu(k))_{k\geq 0}$ and $\overline\nu=(\overline\nu(k))_{k\geq 0}$ be the tail measures of $\mu$ and $\nu$, i.e. for $k\geq 0$,
	\begin{equation*}
	\overline\mu(k)= 1-\sum_{0\leq j\leq k}\mu(j), \quad  \overline\nu(k)=1-\sum_{0\leq j\leq k}\nu(j),
	\end{equation*}
	where we have set $\nu(0):=0$. Define the function $\Delta\scale:E\to[0,\infty)$ recursively by $\Delta \scale(0)=\frac1{\alpha\mu(0)}$ and for $k\geq 0$,
	\begin{equation*}
	\Delta \scale(k+1) = \frac1{\mu(0)}\sum_{j=0}^k \Delta \scale(j) \left( \overline \mu(k-j+1) + \frac p \alpha \right).
	\end{equation*}
	Define the functions $\scale,\kappa:E\to[0,\infty)$ by
	\begin{equation*}
	\scale(k)=\sum_{j=0}^k\Delta \scale(j) \quad \text{and} \quad \kappa(k)= \beta(\Delta \scale\star \overline\nu)(k) + q \scale(k), \quad k\geq 0.
	\end{equation*}
	Define the nonnegative measure $\pi=(\pi(k))_{k\geq 0}$ and the signed measure $\varpi=(\varpi(k))_{k\geq 0}$ recursively by $\pi(0)=\varpi(0)=1$ and for $k\geq 0$,
	\begin{equation*}
	\pi(k+1) = \frac1{k+1}(\pi\star \kappa)(k)  \quad \text{and} \quad \varpi(k+1) = -\frac1{k+1} (\varpi \star\kappa)(k).
	\end{equation*}
Then the GFs of $\Delta\scale$, $\scale$, $\kappa$, $\pi$ and $\varpi$ are given by, for $s\in[0,s_0)$,
\begin{align*}
\Delta\scale^*[s] = & \frac{1-s}{\psi(s)}
&\scale^*[s]= & \frac1{\psi(s)} &
\kappa^*[s] = & \frac{\phi(s)}{\psi(s)}   \\
\pi^*[s] = & \exp \left( \int_0^s \frac{\phi(u)}{\psi(u)} d u \right)
& \varpi^*[s] = &  \exp \left(- \int_0^s \frac{\phi(u)}{\psi(u)} d u \right).
\end{align*}	
\end{lemma}
\begin{proof}
	Let $f:E\to [0,\infty)$ with $f(0)=0$ be given and let $f^{\star n}$ denote the $n$-fold convolution of $f$ with $f^{\star0}(x)=\delta_{0x}$. It is well-known that, given $u_0\geq 0$, $u(x)=u_0\sum_{n\geq 0} f^{\star n}(x)$, is well-defined and satisfies the renewal sequence
	\begin{equation}\label{renewalseq}
	u(x+1)=\sum_{j=0}^x u(j)f(x+1-j), \quad x\geq 0, \quad u(0)=u_0,
	\end{equation}
	see e.g. Theorem V.2.4 in \cite{ASM03} for the arguments. So if $f^*[s]<1$ for some $s>0$, then $u^*[s]=\frac{u_0}{1-f^*[s]}<\infty$. Let $\eta(x)=\frac1{\mu(0)}(\overline \mu(x) + \frac p \alpha)$ for $x\geq 1$ and set $\eta(0)=0$. Since, for $s\in[0,1)$, $\overline\mu^*[s]=\frac{1-\mu^*[s]}{1-s}$ we have
	\begin{equation*}
	\eta^*[s] = \frac1{\mu(0)} \left( \frac{1-\mu^*[s]}{1-s} - (1-\mu(0)) + \frac p{\alpha(1-s)} \right)= 1 -\frac{\psi(s)}{\alpha\mu(0)(1-s)}.
	\end{equation*}
	So $\eta^*[s]<1$ for $s\in[0,s_0)$ and so we conclude that for $s\in[0,s_0)$, $\Delta\scale^*[s] = \frac{\Delta\scale(0)}{1-\eta^*[s]}=\frac{1-s}{\psi(s)}$.
	Consequently, $\scale^*[s]= \frac{\Delta\scale^*[s]}{1-s}= \frac1{\psi(s)}$  for $s\in[0,s_0)$. Since $\overline\nu^*[s]=\frac{1-\nu^*[s]}{1-s}$, we have
	$\kappa^*[s]= \beta \left( \frac{1-s}{\psi(s)} \frac{1-\nu^*[s]}{1-s} \right) + \frac q{\psi(s)}=\frac{\phi(s)}{\psi(s)}$ for $s\in[0,s_0)$.
	From the definitions of $\pi$ and $\varpi$ it follows that $|\varpi(x)|\leq \pi(x)\leq u(x)$ where $u$ is given by the renewal sequence \eqref{renewalseq} with $f(j)=\kappa(j-1)$, $j\geq 1$. Since $f^*[s]=\kappa^*[s]s$ there exists $\epsilon>0$ such that $f^*[s]<1$  and consequently $u^*[s]<\infty$ for $s\in[0,\epsilon]$. Hence  $\pi^*[s]$ and $\varpi^*[s]$ are well-defined for $s\in[0,\epsilon]$.
	By the recursive definitions of $\pi$ and $\varpi$,
	$\frac{d}{d s}\pi^*[s] = \pi^*[s]\frac{\phi(s)}{\psi(s)}$ with $\pi^*[0]=1$ and $\frac{d}{d s}\varpi^*[s] = -\varpi^*[s]\frac{\phi(s)}{\psi(s)}$ with $\varpi^*[0]=1$. By solving these ODEs we get $\pi^*[s]= \exp \left( \int_0^s \frac{\phi(u)}{\psi(u)}d u \right) =\frac1{\varpi^*[s]}$ for $s\in[0,\epsilon]$. Since the middle part of the last equation is well-defined for all $s\in[0,s_0)$ it follows by properties of power series that this last equation holds for all $s\in[0,s_0)$.
\end{proof}

The next two lemmas are more of less covered in Section 2 of Li et al. \cite{li_chen_pakes_2012}, though we provide a  proof since \cite{li_chen_pakes_2012} has a different starting definition of MBIs and uses different notation.
\begin{lemma}\label{lem_MBI_MGF}
Define $\Lambda:[0,s_0)\to[0,\infty)$ by $\Lambda(s)=\int_0^s \frac1{\psi(u)}\mathrm d u$. The function $\Lambda$ is bijective and so the inverse $\Lambda^{-1}:[0,\infty)\to[0,s_0)$ is well-defined.
Then, for $x\in E$, $t\geq 0$ and $s\in[0,s_0)$, the GF of the MBI $(X,\mathbb P)$ is given by,
\begin{equation}\label{GF_MBI}
\mathbb E_x \left[ s^{X_t}\1_{\{t<\zeta\}} \right] = \Psi_t(s)^x \Phi_t(s), \quad
\end{equation}
where
\begin{align}
\Psi_t(s) = & \Lambda^{-1}(t+\Lambda(s)) \label{Psi} \\
\Phi_t(s) = & {\pi^*[s] }{ \varpi^*\left[\Psi_t(s)\right] } \label{Phi}.
\end{align}
\end{lemma}
\begin{proof}
	Note that $\Lambda$ is bijective since it is increasing, continuous, $\Lambda(0)=0$ and, as $\psi$ is a power series with $\psi(s_0)=0$, $\lim_{s\uparrow s_0}\Lambda(s)=\infty$.
	The first part of the proof of Theorem 2.2 in \cite{li_chen_pakes_2012} in combination with p.83 in \cite{anderson} shows that, for $x,y\in E$ and $t\geq 0$, $P_t(x,y)$ is the unique nonnnegative solution to the forward equations  satisfying $\sum_{y\in E}P_t(x,y)\leq 1$. Let $\widetilde P_t(x,y)= \left( P_t(0,\cdot)\star P_t^{\circ}(x,\cdot) \right) (y)$ where $P_t^{\circ}(x,y)$ is the transition kernel associated with the MBP with parameters $(\alpha,\mu,p)$ whose Q-matrix we denote by $Q^{\circ}$. Clearly, $\widetilde P_t(x,y)\geq 0$ and $\sum_{y\in E}\widetilde P_t(x,y)\leq 1$. Since $P_t$, respectively $P^{\circ}_t$, satisfies the forward equations associated with $Q$, respectively $Q^{\circ}$ (see e.g. Section 2.8 in \cite{norris}), we have
	\begin{equation*}
	\begin{split}
	\frac{d}{d t} \widetilde P_t(x,y) = & \sum_{l=0}^{y+1} \sum_{k=0}^{y+1} P_t(0,l)P_t^{\circ}(x,k) \left(  Q(l,y-k) +   Q^{\circ}(k,y-l) \right) \\
	= & \sum_{l=0}^{y+1} \sum_{k=0}^{y+1} P_t(0,l)P_t^{\circ}(x,k)   Q(l+k,y) \\
	= & \sum_{m=0}^{y+1} \sum_{l=0}^m P_t(0,l)P_t^{\circ}(x,m-l) Q(m,y) \\
	= & \sum_{m\in E}\widetilde  P_t(x,m)Q(m,y),
	\end{split}
	\end{equation*}
	where the second equality is due to \eqref{QmatMBI} and where we have set $Q(x,-1)=Q^{\circ}(x,-1)=0$. So $\widetilde P_t$ is a solution to the forward equations and therefore $P_t=\widetilde P_t$. From Theorem 3.1(2) in Anderson or Theorem V.4.1 in Harris \cite{harris}\footnote{In these references it is assumed that $p=0$ but the arguments go through for $p>0$.}, we get that $P_t^{\circ}$ satisfies the branching property $P_t^{\circ}(x,y) = P_t^{\circ}(1,\cdot)^{\star x}(y)$ where $\star x$ denotes $x$-fold convolution. Combining this with $P_t(x,y)= \left( P_t(0,\cdot)\star P_t^{\circ}(x,\cdot) \right) (y)$, we see that \eqref{GF_MBI} holds with $\Psi_t(s)=P_t^{\circ}(1,\cdot)^*[s]$ and $\Phi_t(s)=P_t(0,\cdot)^*[s]$. Since $P_t$, respectively $P^{\circ}_t$, also satisfies the backward equations associated with $Q$, respectively $Q^{\circ}$, we have with the help of Fubini, for $s\in[0,1]$ and $t\geq 0$,
	\begin{equation*}
	\begin{split}
P_t^{\circ}(1,\cdot)^*[s] = &  s + \int_0^t \sum_{z\in E} Q^{\circ}(1,z)P^{\circ}_u(z,\cdot)^*[s] d u \\  %
 P_t(0,\cdot)^*[s] = &  1+ \int_0^t \sum_{z\in E} Q(0,z)P_u(z,\cdot)^*[s] d u . %
 \end{split}
	\end{equation*}
So by taking derivatives in $t$ and using \eqref{conn_Qmat_mech}, we get, for $s\in[0,1]$ and $t\geq 0$,
\begin{align}
\frac{d}{d t} \Psi_s(t) = & \psi(\Psi_t(s)), & \Psi_s(0)= s, \label{Psi_ode} \\
\frac{d}{d t} \Phi_s(t) = & -\phi(\Psi_t(s))\Phi_s(t), & \Phi_s(0)= 1 . \label{Phi_ode}
\end{align}
Since $\psi$ is locally Lipschitz on $(-1,1)$ and, for any $s\in[0,s_0)$, the right-hand side of \eqref{Psi} solves the ODE \eqref{Psi_ode} for all $t\geq 0$, it follows by basic ODE theory that \eqref{Psi} holds. Further, the linear ODE \eqref{Phi_ode} has a unique solution and via \eqref{Psi_ode} and Lemma \ref{lem_recursions} the right-hand side  of \eqref{Psi} solves \eqref{Phi_ode} for all $t\geq 0$ for any given $s\in[0,s_0)$,  which implies that \eqref{Phi} holds.
\end{proof}

\begin{lemma}\label{lem_stat_MBI}
Tbe nonnegative measure $\pi$ defined in Lemma \ref{lem_recursions} is a stationary measure for $(X,\mathbb P)$.
\end{lemma}
\begin{proof}
	We need to prove $\sum_{x\in E}\pi(x)P_t(x,y)= \pi(y)$, for all $y,t\geq 0$ and $t\geq 0$. We do this by showing the GFs in $y$ of both sides are equal. Fix $t\geq 0$ and $s\in[0,s_0)$. We have by respectively Tonelli, \eqref{GF_MBI},  \eqref{Psi} (which implies $\Psi_t(s)\in[0,s_0)$) and \eqref{Phi} (recalling $\pi^*[s]=1/\varpi^*[s]$),
	\begin{equation*}
	\begin{split}
	\sum_{y\geq 0}\sum_{x\in E}\pi(x)P_t(x,y) s^y = & \sum_{x\in E}\pi(x) \sum_{y\geq 0} P_t(x,y) s^y
	=   \sum_{x\in E}\pi(x)  \Psi_t(s)^x \Phi_t(s)
	=  \pi^*[ \Psi_t(s) ] \Phi_t(s)
	= \pi^*[s].
	\end{split}
	\end{equation*}
\end{proof}

Let us mention the following about transience and recurrence for MBIs.
If $\beta=q=0$, i.e.~$(X,\mathbb P)$ is an MBP, then $0$ is an absorbing state and thus recurrent whereas all other states are transient (as $0$ is attainable from any other state). If $(X,\mathbb P)$ is not an MBP, then either $q>0$ in which case all states are transient or $\beta>0$ and $q=0$ in which case the state space is irreducible and so either all states are transient or all states are recurrent. In the latter case results on whether the state space is transient or recurrent can be found in Section 3 in \cite{li_chen_pakes_2012}. So there are three possible cases: either all states are transient, all states are recurrent or $(X,\mathbb P)$ is an MBP. The following theorem provides a full set of fluctuation identities for MBIs.

\begin{theorem}\label{thm_MBI}
	Assume $(X,\mathbb P)$ is an MBI with parameters  $(\alpha,\mu,p,\beta,\nu,q)$. Let for $x,y\geq 0$,
	\begin{equation*}
	\Hnormy(x) = \sum_{l=0}^{y-x-1} \pi(y-x-1-l)\frac{(\scale\star\varpi)(l)}{l+x+1},
	\end{equation*}
where we understand that the sum equals $0$ if $y\leq x$. Then for $x\geq 0$ and $s\in[0,s_0)$,
\begin{equation*}
\sum_{y\geq 0} 	\Hnormy(x)  s^y
= \int_0^{s}  \frac{v^x}{\psi(v)}\exp \left( \int_v^s \frac{\phi(u)}{\psi(u)} d u \right)  d v.
\end{equation*}
If $\beta=q=0$, then $\Hnormy(x)=\frac{\scale(y-x-1)}y$ where we set $\scale(z)=0$ for $z\leq -1$.
We further have the following.
\begin{enumerate}[label=(\arabic*),ref=(\arabic*)]
	\item\label{item_MBI_resolv} If all states of $(X,\mathbb P)$ are transient, then the resolvent of $(X,\mathbb P)$ is  given by, for $x,y\geq 0$,
	\begin{equation*}
	\begin{split}
	G(x,y) %
	= \pi(y) \int_0^{s_0}  \frac{v^x}{\psi(v)} \exp \left(- \int_0^v \frac{\phi(u)}{\psi(u)} d u \right) d v - \Hnormy(x), \\
	\end{split}
	\end{equation*}
	whereas if $(X,\mathbb P)$ is an MBP, then, for $x\geq 0$ and $y\geq 1$,
	\begin{equation*}
	\begin{split}
	G(x,y) %
	=  \frac1y \left( s_0^x \scale(y-1) - \scale(y-x-1) \right).
	\end{split}
	\end{equation*}

	\item\label{item_MBI_htp} If all states are transient, then for $x,y\geq 0$,
	\begin{equation*}
	\mathbb P_x(T_y<\zeta)  %
	= \frac{\pi(y)\int_0^{s_0}  \frac{v^x}{\psi(v)} \exp \left(- \int_0^v \frac{\phi(u)}{\psi(u)} d u \right)  d v -  \Hnormy(x) }{\pi(y) \int_0^{s_0} \frac{ v^y}{\psi(v)} \exp \left(- \int_0^v \frac{\phi(u)}{\psi(u)} d u \right)  d v },
	\end{equation*}
	whereas if $(X,\mathbb P)$ is an MBP, then for $x\geq 0$,
	\begin{equation*}
	\mathbb P_x(T_y<\zeta)  %
	=
	\begin{cases}
	\frac{s_0^x\scale(y-1) -\scale(y-x-1)}{s_0^y\scale(y-1)} & \text{if $y\geq 1$}, \\
	s_0^x  & \text{if $y=0$}.
	\end{cases}
	\end{equation*}
	Clearly,  $\mathbb P_x(T_y<\zeta)=1$ for all $x,y\geq 0$ if all states are recurrent.
		
	\item\label{item_MBI_fluc} In any case, for $0\leq a\leq x\leq b-1$,
	\begin{equation*}
	\begin{split}
	\mathbb P_x(T_{a}<T_{[b}\wedge\zeta) %
	= & \frac{ \Hnormb(x) }{ \Hnormb(a) }, \\
	\mathbb P_x (T_{(a,b)^\complement} <\zeta)
	= &
	1 - \sum_{z=a+1}^{b-1} (pz +q) \left( \frac{\Hnormb(x)}{\Hnormb(a)}\Hnormz(a) - \Hnormz(x) \right), \\
	\mathbb P_x (T_{[b} <\zeta) %
	= &
	\begin{cases}
	1 - \sum_{z=0}^{b-1} (pz +q) \left( \frac{\pi(z)}{\pi(b)}\Hnormb(x) - \Hnormz(x) \right) & \text{if $\beta+q>0$}, \\
	1  +  p\sum_{z=0}^{b-x-2}\scale(z) - \frac{\scale(b-x-1)}{\scale(b-1)} \left( 1 + p \sum_{z=0}^{b-2}\scale(z) \right)   & \text{if $\beta=q=0$}.
	\end{cases}
	\end{split}
	\end{equation*}
\end{enumerate}		
\end{theorem}

\begin{proof}
We have for $x\geq 0$ and $s\in[0,s_0)$,
\begin{equation}\label{GF_H_MBI}
\begin{split}
\sum_{y\geq 0} 	\Hnormy(x)  s^y =  \sum_{y\geq x+1} 	\Hnormy(x)  s^y
= & \pi^*[s] \sum_{l\geq 0}\frac{(\scale\star\varpi)(l)}{l+x+1}   s^{x+1+l} \\
= & \pi^*[s] \int_0^s  \sum_{l\geq 0} (\scale\star\varpi)(l)  v^{x+l} d v \\
= & \pi^*[s] \int_0^s v^x\scale^*[v]\varpi^*[v] d v,
\end{split}
\end{equation}	
where we used Fubini for the penultimate equality which is applicable because clearly $|\varpi(l)|\leq \pi(l)$ for all $l\geq 0$ by their definitions. So the first statement follows from Lemma \ref{lem_recursions}.	If $\beta=q=0$, then by Lemma \ref{lem_recursions} we easily see that $\pi(0)=\varpi(0)=1$ and $\pi(l)=\varpi(l)=0$ for $l\geq 1$ and consequently $\Hnormy(x)=\frac{\scale(y-x-1)}y$. Next we prove item (1). Assume all states are transient so that $G(x,y)<\infty$ for all $x,y\geq 0$. For $x\geq 0$ and $s\in[0,s_0)$ we have by Lemma \ref{lem_MBI_MGF} and the change of variables $t=\Lambda(v)-\Lambda(s)$,
\begin{equation*}
\begin{split}
\sum_{y\geq 0} G(x,y)s^y = \int_0^\infty P_t(x,y)s^y d t
=  & \int_0^\infty \left( \Lambda^{-1}(t+\Lambda(s)) \right)^x \pi^*[s] \varpi^*\left[ \Lambda^{-1}(t+\Lambda(s))\right]   d t \\
= & \pi^*[s]\int_s^{s_0}  \frac{v^x\varpi^*[v]}{\psi(v)}  \mathrm d v \\
= & \pi^*[s] \int_0^{s_0}  \frac{v^x\varpi^*[v]}{\psi(v)}  \mathrm d v - \pi^*[s] \int_0^s  \frac{v^x\varpi^*[v]}{\psi(v)}  d v.
\end{split}
\end{equation*}
So the expression for $G(x,y)$ follows by \eqref{GF_H_MBI} and Lemma \ref{lem_recursions} in the case where all states are transient.
Now assume $(X,\mathbb P)$  is an MBP, i.e. $\beta=q=0$. Let, for $r>0$, $(X,\mathbb P^{(r)})$ be the MBI with parameters $(\alpha,\mu,p,0,\nu,r)$. Then via \eqref{GF_MBI}, the transition kernel of $(X,\mathbb P^{(r)})$ is given by $P^{(r)}_t(x,y)=  e^{-r t }P_t(x,y)$, $x,y\in E$, $t\geq 0$.
Moreover, by the same reasoning as for \eqref{killed_process_stoptime}, for any stopping time $T$,
\begin{equation}\label{feynman-kac_const}
\mathbb P_x^{(r)}(\Gamma\cap\{T<\infty\})=\mathbb E_x \left[ e^{-r T } \1_{\Gamma\cap\{T<\infty\}}  \right], \quad \Gamma\in\mathcal F_T, \ x\geq 0.
\end{equation}
Let $G^{(r)}(x,y)=\int_0^\infty P_t^{(r)}(x,y) d t$ be the resolvent of the MBI with parameters $(\alpha,\mu,p,0,\nu,r)$ and denote by $\pi_r$ and $\varpi_r$ the measures defined by $\pi_r(0)=\varpi_r(0)=1$ and, for $k\geq 0$, $\pi_r(k+1)=\frac r{k+1}(\pi_r\star\scale)(k)$ and $\varpi_r(k+1)=-\frac r{k+1}(\varpi_r\star\scale)(k)$. Since for the MBI with parameters $(\alpha,\mu,p,0,\nu,r)$ all states are transient as $r>0$, we can use the formula for the resolvent that we just proved to deduce for $x,y\geq 0$,
\begin{equation*}
 G^{(r)}(x,y) =  \pi_r(y) \int_0^{s_0}  \frac{v^x}{\psi(v)} \exp \left(- \int_0^v \frac{r}{\psi(u)} d u \right) d v - \sum_{l=0}^{y-x-1} \pi_r(y-x-1-l)\frac{(\scale\star\varpi_r)(l)}{l+x+1}.
\end{equation*}
It is easy to prove by induction that $\lim_{r\downarrow 0}\pi_r(y)=\lim_{r\downarrow 0} \varpi_r(y)=0$ and $\lim_{r\downarrow 0}\frac{\pi_r(y)}r=\frac{\scale(y-1)}y$ for $y\geq 1$ and further by an integration by parts, for $x\geq 0$,
\begin{multline}\label{IBP_MBP_resol}
\lim_{r\downarrow 0} r  \int_0^{s_0}  \frac{v^x}{\psi(v)} \exp \left(- \int_0^v \frac{r}{\psi(u)} d u \right) d v \\
= \lim_{r\downarrow 0}\left( 1 + \int_0^{s_0}  x v^{x-1} \exp \left(- \int_0^v \frac{r}{\psi(u)} d u \right)  d v \right) =s_0^x.
\end{multline}
So by the monotone convergence theorem we deduce, for all $x\geq 0$ and $y\geq 1$,
\begin{equation*}
\begin{split}
G(x,y)=\lim_{r\downarrow 0} G^{(r)}(x,y) = & \lim_{r\downarrow 0} \frac{\pi_r(y)}r r  \int_0^{s_0} \frac{v^x}{\psi(v)} \exp \left(- \int_0^v \frac{r}{\psi(u)} d u \right) d v - \frac{\scale(y-x-1)}{y} \\
= & \frac{s_0^x\scale(y-1)}y  - \frac{\scale(y-x-1)}{y},
\end{split}
\end{equation*}
which finishes the proof of item \ref{item_MBI_resolv}.

For items \ref{item_MBI_htp} and \ref{item_MBI_fluc} we use the results of Section \ref{sec_pot_fluc}. First assume that all states are transient. Then $\beta>0$ or $q>0$ and thus $\pi$ from Lemma \ref{lem_recursions} is a positive (not merely nonnegative) measure and by Lemma \ref{lem_stat_MBI} it is a stationary and thus excessive measure for $(X,\mathbb P)$. So we can take $\pi$ as our reference measure. We further let the reference point be $\mfo=0$. As $\pi$ is stationary, $\mathbb P_x(t<\zeta)=\sum_{y\geq 0}\widehat P_t(x,y)=\frac{\pi(x)P_t(x,y)}{\pi(y)}=1$ for all $x\geq 0$ and any $t\geq 0$, so the dual Markov chain $(X,\widehat{\mathbb P})$ has infinite lifetime. Because all states are transient for the dual and the dual is skip-free upwards, it follows that $\widehat{\mathbb P}_y(T_b<\zeta)=1$ for all $y\leq b$. Consequently, by the analogue of \eqref{fpt_above_general} for the dual we deduce $\widehat H(y)=1$ for all $y\geq 0$. By the definition of $H$ in Theorem \ref{thm_fundexc}\ref{item_H}, $H(x)=\frac{G(x,0)}{G(0,0)}$.
Then by \eqref{resolv_expr} and the expression for $G(x,y)$ from item \ref{item_MBI_resolv},
\begin{equation*}
\Hqy(x)=H(x)-\frac{G(x,y)}{\con\widehat H(y)\pi(y)} = \frac1{\con} \left( G(x,0) - \frac{G(x,y)}{\pi(y)}\right)   = \frac{\Hnormy(x)}{\con\pi(y)}, \quad x,y\geq 0.
\end{equation*}
The identities in items \ref{item_MBI_htp}-\ref{item_MBI_fluc} then immediately follow from the identities in Theorem \ref{thm_fluc} in the case where all states are transient. When not all states are transient for $(X,\mathbb P)$, we consider for $r>0$, the MBI $(X,\mathbb P^{(r)})$ for which all states in $E$ are transient and so for which the identities in items \ref{item_MBI_htp}-\ref{item_MBI_fluc} apply. Then via \eqref{feynman-kac_const} and the monotone convergence theorem we can take $r\downarrow 0$ on both sides of the identities for $(X,\mathbb P^{(r)})$ to get the identities for $(X,\mathbb P)$ using that, with the obvious notation, $\lim_{r\downarrow 0}\pi_r(x)=\pi(x)$ and $\lim_{r\downarrow 0}\Hnormy_r(x)=\Hnormy(x)$ for $x\geq 0$ and using in the case where $\beta=q=0$, $\lim_{r\downarrow 0}\frac{\pi_r(y)}r=\frac{\scale(y-1)}y$ for $y\geq 1$ and \eqref{IBP_MBP_resol}.
\end{proof}

\begin{remark}\label{remark_MBI_killingparam}
	Regarding the role of the killing parameters $p$ and $q$, if $(X,\mathbb P)$ is an MBI with parameters  $(\alpha,\mu,p,\beta,\nu,q)$ and $(X,\mathbf P)$ an MBI with parameters $(\alpha,\mu,0,\beta,\nu,0)$, then following the arguments leading up to \eqref{killed_process_stoptime} one can show for any stopping time $T$,
	\begin{equation*}
	\mathbb P_x(\Gamma\cap\{T<\infty\})=\mathbf E_x \left[ e^{-q T -p\int_0^T X_s d s} \1_{\Gamma\cap\{T<\infty\}}  \right], \quad \Gamma\in\mathcal F_T, \ x\geq 0,
	\end{equation*}
	where $\mathbf E_x$ is the expectation operator associated with $\mathbf P_x$. So all the identities in Theorem \ref{thm_MBI} can be rephrased in terms of the MBI $(X,\mathbf P)$ for which the killing parameters $p$ and $q$ are $0$. As an example,
	\begin{equation*}
 \mathbf E_x \left[  e^{-q T_a - p\int_0^{T_a} X_s  d s} \1_{\{T_{a}<T_{[b}\wedge\zeta\}} \right] = \mathbb P_x(T_{a}<T_{[b}\wedge\zeta)  =
	\frac{ \Hnormb(x) }{ \Hnormb(a) }, \quad 0\leq a\leq x\leq b-1.
	\end{equation*}
	A similar remark holds of course for the role of the killing parameter $p$ for downward skip-free compound Poisson processes considered in Section \ref{sec_CPP}.
\end{remark}

Most of the identities in Theorem \ref{thm_MBI} are new including in the sense that no analogues have appeared in the literature for the discrete-time analogue of MBIs, i.e.~Galton-Watson processes, or the continuous-state space analogue of MBIs, i.e.~continuous-state branching processes with immigration (CBIs). For the special case of MBPs some of these identities are covered in the literature because the Lamperti transform tells us that an MBP is a time-changed skip-free downward compound Poisson process stopped at hitting 0 and killed at a constant rate, see Theorem III.11.3 in \cite{Ney-Bran} for the Lamperti transform for nondecreasing MBPs. This implies that, in the MBP case, the first two identities in Theorem \ref{thm_MBI}\ref{item_MBI_fluc} and the identity in Theorem \ref{thm_MBI}\ref{item_MBI_htp} for $y\leq x$ are identical to those for skip-free downward compound Poisson processes, which is confirmed by comparing Theorems \ref{thm_CPP} and \ref{thm_MBI}, see also Theorem 12.8 and Corollary 12.9 in \cite{Kyprianou} for the continuous-state space case. With immigration the only identity that, to our knowledge, has appeared before in the literature is the one for the downwards hitting probability, which from Theorem \ref{thm_MBI}\ref{item_MBI_htp} reads, in the case where all states are transient,
\begin{equation}\label{hitprob_downw_MBI}
\mathbb P_x(T_a<\zeta)  %
= \frac{ \int_0^{s_0}  \frac{v^x}{\psi(v)} \exp \left(- \int_0^v \frac{\phi(u)}{\psi(u)} d u \right)  d v }{  \int_0^{s_0} \frac{ v^a}{\psi(v)} \exp \left(- \int_0^v \frac{\phi(u)}{\psi(u)} d u \right)  d v }, \quad 0\leq a\leq x.
\end{equation}
Namely, the analogue of this identity for CBIs has been derived  by Duhalde et al. in \cite{DFM} and the MBI case can be found in \cite{Vidmar}, see also \cite{paper1} for the case $\beta=0$. We remark that given the input parameters $(\alpha,\mu,p,\beta,\nu,q)$ and given $y\geq 0$ the object $\Hnormy(x)$, $0\leq x\leq y-1$, can be computed exactly in a finite number of steps via the recursions in Lemma \ref{lem_recursions}. On the other hand for e.g. the downwards hitting probability \eqref{hitprob_downw_MBI} one needs to evaluate the branching and immigration mechanisms $\psi$ and $\phi$ which might not always have a fully explicit form depending on the input parameters. Since $\mathbb P_x(T_a<\zeta)=\lim_{b\to\infty} \mathbb P_x(T_{a}<T_{[b}\wedge\zeta)$ one can alternatively approximate  $\mathbb P_x(T_a<\zeta)$ by evaluating $\mathbb P_x(T_{a}<T_{[b}\wedge\zeta)$ for a sufficiently large $b$ which avoids needing to compute $\psi$ and $\phi$ or the integrals in \eqref{hitprob_downw_MBI}.
In Theorem \ref{thm_MBI} we see that the expressions simplify for the special case of MBPs. This is also the case  for another subclass of MBIs.

\begin{example}
Assume the parameters $(\alpha,\mu,p,\beta,\nu,q)$ of the MBI are such that $0<m\leq 1+\frac p\alpha$, $\beta=\alpha m$, $\nu(j)=\frac{(j+1)\mu(j+1)}m$ for $j\geq 1$ and $q=p-\alpha(m-1)$, where $m:=\sum_{j\geq 2}j\mu(j)$. Then $\phi(s)=-\psi'(s)$ for all $s\geq 0$ and consequently, for $s\in[0,s_0)$ and $y\geq 0$,  via Lemma \ref{lem_recursions} and Theorem \ref{thm_MBI},
\begin{equation*}
\int_0^{s_0}  \frac{v^x}{\psi(v)} \exp \left(- \int_0^v \frac{\phi(u)}{\psi(u)} d u \right) d v = \frac1{\psi(0)}\frac{s_0^{x+1}}{x+1}, \quad \pi^*[s] = \frac{\psi(0)}{\psi(s)}, \quad  \sum_{y\geq 0} \Hnormy(x)  s^y = \frac1{\psi(s)}\frac{s^{x+1}}{x+1}.
\end{equation*}
This yields, for $x,y\geq 0$,
\begin{equation*}
\pi(x)= \psi(0)W(x), \quad \Hnormy(x)= \frac{W(y-x-1)}{x+1}.
\end{equation*}
\end{example}	

\bigskip

\bibliographystyle{plain}

\end{document}